\documentclass[11pt]{amsart}

\usepackage[margin=1in]{geometry}  
\usepackage{graphicx}              
\usepackage{amsmath}               
\usepackage{amsfonts}              
\usepackage{amsthm}                
\usepackage{color}
\usepackage{tabulary}
\usepackage{caption}
\usepackage{subcaption}
\usepackage{amssymb}
\usepackage{todonotes}
\theoremstyle{plain}
\usepackage{mathptmx}

\newtheorem{thm}{Theorem}[section]

\newtheorem{lem}[thm]{Lemma}

\newtheorem{prop}[thm]{Proposition}
\newtheorem{cor}[thm]{Corollary}
\newtheorem{conj}[thm]{Conjecture}
\newtheorem{defn}[thm]{Definition}
\newtheorem{rmk}[thm]{Remark}

\newtheorem{ques}[thm]{Question} 

\theoremstyle{definition}
\newtheorem{thmx}{Theorem}

\DeclareMathOperator{\PR}{PSL_2(\RR)}
\DeclareMathOperator{\PC}{PSL_2(\CC)}

\DeclareMathOperator{\Homeop}{Homeo^+}

\DeclareMathOperator{\Fix}{Fix}

\DeclareMathOperator{\Per}{Per}

\newcommand{\ZZ}{\mathbb{Z}}      
\newcommand{\RR}{\mathbb{R}}      
\newcommand{\CC}{\mathbb{C}}      
\newcommand{\DD}{\mathbb{D}}      
\newcommand{\HH}{\mathbb{H}}      

\newcommand{\ie}{\emph{i.e.}}

\newcommand{\x}{\overline{x}}
\newcommand{\y}{\overline{y}}
\newcommand{\z}{\overline{z}}

\begin{document}

\title[Laminar groups, Tits alternatives, and convergence group actions on $S^2$]{On laminar groups, Tits alternatives, \\ and convergence group actions on $S^2$}

\author{Juan Alonso}
\address{Centro de Matem\'atica, Facultad de Ciencias, Universidad de la Rep\'ublica. Igu\'a 4225, Montevideo C.P. 11400, Uruguay.}
\email{juan@cmat.edu.uy}

\author{Hyungryul Baik} 
\address{%
		Department of Mathematical Sciences, KAIST\\
		291 Daehak-ro Yuseong-gu, Daejeon, 34141, South Korea 
}
\email{%
        hrbaik@kaist.ac.kr
}

\author{Eric Samperton}
\address{Department of Mathematics, \\ 
South Hall, Room 6607, University of California \\
Santa Barbara, CA 93106-3080}
\email{eric@math.ucsb.edu}

\begin{abstract}
Following previous work of the second author, we establish more
properties of groups of circle homeomorphisms which admit invariant
laminations.  
In this paper, we focus on a certain type of such groups-so-called
pseudo-fibered groups, and show that many 3-manifold groups are
examples of pseudo-fibered groups. 
We then prove that torsion-free pseudo-fibered groups satisfy a Tits
alternative. 
We conclude by proving that a purely hyperbolic pseudo-fibered group
acts on the 2-sphere as a convergence group. 
This leads to an interesting question if there are examples of
pseudo-fibered groups other than 3-manifold groups. 



\smallskip
\noindent \textbf{MSC classes:} 20F65, 20H10, 37C85, 37E10, 57M60. 
\end{abstract}

\maketitle

\section{Introduction}\label{intro}
Thurston \cite{Thurston97} showed that if $M$ is an atoroidal
3-manifold admitting a taut foliation, then $\pi_1(M)$ acts faithfully
on $S^1$ with a pair of dense invariant laminations. The result was
generalized by Calegari-Dunfield \cite{CalDun03} and one can find a
complete treatment in \cite{Calebook}. Motivated by these results, the
second author studied groups acting faithfully on $S^1$ with
prescribed types and numbers of invariant laminations, in the process
giving a new characterization of Fuchsian groups
\cite{BaikFuchsian}. In the same paper, he asked if there exists a way
of characterizing fibered 3-manifold groups in a similar way. 

For a 3-manifold $M$ which fibers over the circle, one can construct a
natural action of $\pi_1(M)$ on the circle wth two
invariant laminations and this motivates the following definition.  
We call a subgroup $G$ of $\Homeop(S^1)$ \emph{pseudo-fibered}
if its action on $S^1$ admits two invariant, very full, loose
laminations with distinct endpoints.  
As we said, this includes a large class of examples coming from 3-manifolds which
fiber over the circle, and the purpose of this paper is to show that
pseudo-fibered groups in general have many nice properties. 

We emphasize that in this paper
all group actions on $S^1$ are considered up to conjugacy not
semi-conjugacy (compare \cite{Mann15}). 
This is because it is not clear if the notion of pseudo-fibering is invariant under semi-conjugacy. 
See Section \ref{sec:Prelim} for the precise definitions and relevant discussions.

In Section \ref{sec:Tits} we establish our first main result, a Tits alternative for torsion-free pseudo-fibered groups:

\begin{thmx}
\label{thm:main1}
Let $G$ be a torsion-free pseudo-fibered group. Each subgroup of $G$ either contains a non-abelian free subgroup or is virtually abelian. 
\end{thmx} 

This is proved in Section \ref{subsec:Tits} by studying how two elements of a (torsion-free) pseudo-fibered group interact with each other dynamically. In particular, we show a kind of dynamical alternative for elements of a pseudo-fibered group:

\begin{thmx}
\label{thm:dynamics}
Let $G$ be a torsion-free pseudo-fibered group and for $f \in G$, let
$\Per_f$ denote the set of all periodic points of $f$ on $S^1$. Then
for any $g, h \in G$, $\Per_g$ and $\Per_h$ are either equal or disjoint.
\end{thmx}

Theorem \ref{thm:main1} will follow from Theorem \ref{thm:dynamics} by applying the ping-pong lemma, and H\"{o}lder's theorem that a group acting faithfully and freely on $\RR$ is necessarily abelian.

In Section \ref{sec:purelyhyperbolic}, we study
pseudo-fibered groups with more structure, inspired by (quasi-)Fuchsian groups.
This part should be considered as part of Fenley's program which generalizes the
work of Cannon-Thurston \cite{CannonThurston} considerably from the viewpoint of pseudo-Anosov
flows (see \cite{Fenley12}).  

Our main result of Section \ref{sec:purelyhyperbolic} connects
the pseudo-fibered group action on $S^1$ with a convegence group
action on $S^2$. Note that a similar idea has been carried out in
Fenley's program (see \cite{Fenley12}, \cite{Fenley16}, and also compare \cite{Frankel12}, \cite{BarbotFenley15}). 
\begin{thmx} 
\label{thm:main2}
Let $G$ be a pseudo-fibered group which is purely hyperbolic.  Then $G$ acts on $S^2$ as a convergence group.
\end{thmx}

For the definition of ``purely hyperbolic", see Section
\ref{sec:Prelim}. We expect that Theorem \ref{thm:main2} can be
strengthened.  
Indeed, when $G$ is a purely hyperbolic pseudo-fibered group, the
second author has previously conjectured that $G$ is a Fuchsian group,
hence acts on $S^1$ as a convergence group \cite{BaikFuchsian}. 
Recall that the work of many authors (e.g. \cite{Tukia88},
\cite{CassonJungreis94} and \cite{Gabai91}) 
shows that a group acts on $S^1$ as a convergence group if and only if it is topologically conjugate to a Fuchsian group. 

Theorem \ref{thm:main1}, Theorem \ref{thm:dynamics}, and Theorem
\ref{thm:main2} show that pseudo-fibered groups have
properties similar to those of fibered 3-manifold groups. On the other
hand, we provide a source of examples of pseudo-fibered groups which
are quite different from fibered 3-manifold groups in Section
\ref{sec:promotion}. More precisely, 
in Theorem \ref{thm:free-product} , we show that the free product of any two finite
cyclic groups is a pseudo-fibered group. In the context of
3-manifolds, this implies that
the fundamental group of the connected sum of any two lens spaces is
also pseudo-fibered (Corollary \ref{thm:consum-lensspace}). 

Given the results above and the result of \cite{BaikFuchsian}, we
propose the following 
\begin{conj}[Promotion of Pseudo-Fibering] 
\label{conj:promotion}
Let $G$ be a finitely-generated torsion-free pseudo-fibered group which does not split as a nontrivial free product. Then there are three possibilities. 
\begin{itemize}
\item[1.] $G$ is elementary, \ie virtually abelian.
\item[2.] $G$ is topologically conjugate to a M\"obius group action (as usual, we consider $\PR$ as a subgroup of $\Homeop(S^1)$).
\item[3.] $G$ is abstractly isomorphic to a closed hyperbolic 3-manifold group.
\end{itemize}
\end{conj}

A similar conjecture was made in \cite{BaikFuchsian} and the
difference is discussed in Section \ref{sec:promotion}. Here, $G$ is
said to be elementary if it is virtually abelian. This definition
makes sense due to Theorem \ref{thm:main1}, which asserts that every
non-elementary pseudo-fibered group contains a non-abelian free
subgroup. 

\subsection{Acknowledgements}
We thank David Cohen for asking the second author if the Tits
alternative holds for pseudo-fibered groups. We also wish to thank
Danny Calegari, Ursula Hamenst\"{a}dt, and Dawid Kielak for helpful
conversations. We greatly appreciate for the anonymous referee for the
valuable comments which improved the structure of the paper
significantly. 
The second author was partially supported by the ERC Grant Nb.\
10160104, and Samsung Science \& Technology Foundation grant no. SSTF-BA1702-01.The third author thanks Universit\"{a}t Bonn for
hospitality, during which time some of this work was completed.

\section{Preliminaries}
\label{sec:Prelim}
We briefly review and motivate several definitions regarding laminations on the circle.  Two pairs $(a, b)$ and $(c, d)$ of distinct points of the circle $S^1$ are said to be \emph{linked} if each connected component of $S^1 \setminus \{a, b\}$ contains precisely one of $c, d$. They are called \emph{unlinked} if they are not linked. Let $\mathcal{M}$ denote the set of all unordered pairs of two distinct points of $S^1$, i.e., $\mathcal{M} = (S^1 \times S^1 - \Delta) / (x, y) \sim (y,x)$ where $\Delta$ is the diagonal $\{ (x, x) : x \in S^1\}$. A \emph{lamination} of $S^1$ is a closed subset of $\mathcal{M}$ whose elements are pairwise unlinked. Given a lamination $\Lambda$, an element $(a, b)$ of $\Lambda$ is called a \emph{leaf}, and the points $a, b$ are called the \emph{endpoints} of the leaf $(a,b)$ (or just endpoints of $\Lambda$ if there is no possible confusion).  Two laminations have {\em distinct endpoints} if their sets of endpoints are disjoint.  A lamination $\Lambda$ is called \emph{dense} if the set of endpoints of $\Lambda$ is a dense subset of $S^1$. 

Any subgroup $G$ of $\Homeop(S^1)$ has an induced action on $\mathcal{M}$. We say that a lamination $\Lambda$ is \emph{$G$-invariant} if the $G$-action on $\mathcal{M}$ preserves $\Lambda$ set-wise.   A discrete subgroup $G$ of $\Homeop(S^1)$ is called \emph{laminar} if it admits a dense $G$-invariant lamination. 

Let $\DD$ denote the closed unit disk in $\mathbb{C}$ where the
interior is equipped with the Poincar\'e metric, i.e., $\DD =
\mathbb{H}^2 \cup \partial_\infty \mathbb{H}^2$ .
A lamination $\Lambda'$ of $\DD$ is a set of chords with disjoint
interiors such that there exists a lamination 
$\Lambda$ in $S^1 = \partial \DD$ where the chords in $\Lambda'$ can
be obtained by connecting the endpoints of the leaves of $\Lambda$. 

As noted in Construction 2.4 of \cite{Calebook}, the set of laminations on $S^1$
and the set of geodesic laminations of $\mathbb{H}^2$ are in
one-to-one correspondence up to isotopy relative to $S^1
= \partial_\infty \mathbb{H}^2$. Hence, we freely switch our viewpoint
between these two without further mentioning. 
A \emph{gap} of a lamination $\Lambda$ is the closure of a connected
component of $\mathbb{H}^2 \setminus \Lambda$ in $\mathbb{H}^2 \cup \partial_\infty \mathbb{H}$.  

We recall some key properties of laminations from \cite{BaikFuchsian}
(also compare \cite{Calegarinote}) .

\begin{defn} A lamination $\Lambda$ is said to be 
\begin{itemize}
\item \textbf{totally disconnected} if no open subset of the disk is foliated by $\Lambda$,
\item \textbf{very full} if each gap is a finite-sided ideal polygon in the disk, and
\item \textbf{loose} if no two leaves share an endpoint unless they are edges of the same (necessarily unique) gap.
\end{itemize} 
\end{defn} 

For every element $f$ of $\Homeop(S^1)$, let $\Fix_f \subset S^1$
denote the set of all fixed points of $f$.  Let $\Per_f$ denote the
set of all periodic points of $f$, where a point $p$ of $S^1$ is
\emph{periodic} for $f$ if the orbit of $p$ under $f$ is finite.  Thus
$\Fix_f \subset \Per_f$.  
A fixed point $p$ of the homeomorphism $f$ is \emph{attracting} if
there exists an interval $I \ni p$ containing no other fixed points
such that $f(I) \subsetneq I$. 
Similarly, a fixed point $q$ is \emph{repelling} if there exists an interval $J \ni q$ containing no other fixed points such that $f(J) \supsetneq J$.

We first give names to particular types of homeomorphisms of $S^1$ in
the following definition as in \cite{BaikFuchsian}. For the first
four types of homeomorphisms, compare \cite{Kovacevic99} where
M\"obius-like elliptic, M\"obius-like parabolic, M\"obius-like
hyperbolic, and M\"obius-like homeomorphisms are
defined. 
\begin{defn} 
\label{defn:elementclassification}
An element $f$ of $\Homeop(S^1)$ is said to be 
\begin{itemize}
\item \textbf{elliptic} if $f$ has no fixed points, 
\item \textbf{parabolic} if $f$ has a unique fixed point, 
\item \textbf{hyperbolic} if $f$ has two fixed points, one attracting and one repelling,
\item \textbf{M\"{o}bius-like} if $f$ is conjugate in $\Homeop(S^1)$ to an element of $\PR$, 
\item \textbf{pseudo-Anosov-like} or \textbf{p-A-like} if $f$ is not hyperbolic and some positive power $f^n$ has a positive, even number of fixed points alternating between attracting and repelling, and
\item \textbf{properly pseudo-Anosov-like} or \textbf{properly p-A-like} if $f$ is pseudo-Anosov-like and non-elliptic.  
\end{itemize}
\end{defn}

Thus $f$ is p-A-like if and only if a positive power of $f$ is
properly p-A-like.  For a p-A-like homeomorphism $f \in \Homeop(S^1)$, 
the set of boundary leaves of the convex hull of the attracting fixed
points of a properly p-A-like power $f^n$ is called the 
\textbf{attracting polygon} of $f$.  Similarly, $f$ has a \textbf{repelling polygon}.

\begin{defn}
Let $\Lambda$ be a lamination. A leaf $l \in \Lambda$ is said to be
\textbf{visible} from a point $p \in S^1$ if one can connect $l$ to
$p$ by a geodesic of $\mathbb{H}^2$ (i.e., there exists a geodesic ray from a point
on $l$ to $p$) which does not intersect any leaf of $\Lambda$ in $\mathbb{H}^2$.
\end{defn}

Observe that if $p$ is an endpoint of a gap in a very full, loose
lamination, the set of leaves visible from $p$ is precisely the set of
edges of the gap. 

Now we outline some examples of laminar groups mentioned in the
introduction. Let $S$ be a closed hyperbolic surface, and
$\phi$ be a pseudo-Anosov homeomorphism of $S$. Let $M$ be the
mapping torus $S \times [0,1]/ (x,1) \sim (\phi(x),0)$. Then one can
construct a faithful action of $\pi_1(M)$ on $S^1$ in the following
way. Note $\pi_1(M)$ is isomorphic to $\pi_1(S) \rtimes \ZZ$. The deck
transformation action of $\pi_1(S)$ on $\HH^2$ extends continuously to
a faithful action on $\partial \HH^2$. Let $\widetilde{\phi}: \HH^2
\to \HH^2$ be a lift of $\phi$ to the universal cover of $S$.
Since $S$ has finite area, $\widetilde{\phi}$ is a quasi-isometry,
hence extends to a homeomorphism on $\partial \HH^2$. 
Considering this homeomorphism on $\partial \HH^2$ as a generator of $\ZZ$, this defines an action
$$\rho: \pi_1(M) = \pi_1(S) \rtimes \ZZ \to \Homeop(\partial \HH^2) = \Homeop(S^1).$$
Then $\rho(\pi_1(M))$ is laminar, since it fixes both the stable and unstable laminations of $\phi$.  In fact, $\rho$ is faithful.

\begin{defn}
A finitely generated laminar group $G$ is said to be \textbf{fibered}
if $G$ is topologically conjugate to $\rho(\pi_1(M))$ where $\rho$ and
$M$ are as in the previous paragraph, and the conjugacy takes the
$G$-invariant lamination to one of the invariant laminations of the
monodromy of $M$. 
\end{defn}

\begin{defn}
A finitely generated laminar group $G$ is said to be \textbf{pseudo-fibered} if it preserves a pair of very full loose invariant laminations $\Lambda_1, \Lambda_2$ with distinct endpoints, and each nontrivial element of $G$ has at most countably many fixed points in $S^1$. We also say $(G, \Lambda_1, \Lambda_2)$ is a pseudo-fibered triple. 
\end{defn}

Pseudo-fibered groups were first studied in \cite{BaikFuchsian}, although they had not yet been given a name.  

\begin{thm}[see Section 8 of \cite{BaikFuchsian}] 
\label{thm:pAlike_elements}
Let $(G,\Lambda_1,\Lambda_2)$ be a pseudo-fibered triple.  Let $g \in G$.  Then 
\begin{itemize}
\item[(1)] $g$ is either M\"{o}bius-like or pseudo-Anosov-like, and
\item[(2)] if $g$ is p-A-like, then for some $i,j= 1,2$ with $i \neq j$, $\Lambda_i$ contains the attracting polygon of $g$, and $\Lambda_j$ contains the repelling polygon of $g$.
\end{itemize}
\end{thm}

Furthermore, \cite{BaikFuchsian} shows that if $G$ is torsion-free, then all M\"{o}bius-like elements are hyperbolic elements.

The following proposition justifies the term ``pseudo-fibered''. 

\begin{prop} 
\label{prop:pseudofiberedname} 
A fibered group $G$ is pseudo-fibered. 
\end{prop} 
\begin{proof} We only need to worry about the cardinality of the set
  of fixed points of each element. But this is not a problem due to Theorem 5.5 of
  \cite{casson1988automorphisms} which asserts that for a given
  pseudo-Anosov surface homeomorphism $h$, any lift of a strictly positive power of $h$ has finitely many fixed points on $\partial_\infty \mathbb{H}^2$, alternating between attracting and repelling. 
\end{proof}

In fact the proof of Theorem 5.5 in \cite{casson1988automorphisms}
(pp. 85-87) shows Theorem \ref{thm:pAlike_elements} in the case of
fibered groups. Any lift of a strictly positive power of $h$ falls
into one of the three cases. 
Case 1 and Case 2 correspond to properly pseudo-Anosov-like elements
and Case 3 corresponds to hyperbolic elements in the sense of
Definition \ref{defn:elementclassification}. 
In Case 1, the attracting repelling polygons of the p-A-like element
have 3 or more sides and in Case 2, those polygons are degenerate, 
i.e., there are exactly two attracting fixed points and two repelling fixed points.

We remark that the ``pseudo" in ``pseudo-fibered group" intentionally carries two different connotations.  The first, as in Theorem \ref{thm:pAlike_elements}, indicates that some elements are pseudo-Anosov-like.  The second, as in Conjecture \ref{conj:promotion}, indicates that pseudo-fibered groups are (conjecturally) not far from fibered groups.

In Section \ref{sec:purelyhyperbolic}, we study a special class of pseudo-fibered groups.

\begin{defn}
Let $G$ be a pseudo-fibered group. $G$ is called \textbf{purely hyperbolic} if it has no pseudo-Anosov-like elements.
\end{defn}


Theorem \ref{thm:main2} says that a purely hyperbolic pseudo-fibered
group acts on the sphere as a convergence group. In general, a group
$G$ acting on a compactum $X$ is called a discrete convergence group
if the following holds: for any infinite sequence of distinct elements
$(g_i)$ of $G$, there exists a subsequence $(g_{i_j})$ of $(g_i)$ and
two points $a, b \in X$ not necessarily distinct such that $g_{i_j}$
converges to the constant map with value $a$ uniformly on every
compact subset of $X \setminus \{b\}$, and $g_{i_j}^{-1}$ converges to
the constant map with value $b$ uniformly on every compact subset of
$X \setminus \{a\}$. Since we only deal with discrete convergence
groups in this paper, we will omit the word discrete, and simply call
it a convergence group. 

As mentioned before, when $X$ is $S^1$, being a
convergence group is equivalent to being (conjugate to) a Fuchsian
group. This result is known as the Convergence Group Theorem
\cite{Tukia88,Gabai91,CassonJungreis94}, and the same statement for indiscrete
convergence groups was proved in \cite{Hinkkanen90}. 

\begin{rmk} There is a well-known equivalent definition of a convergence group action. $G$ acting on $X$ is called a convergence group if the diagonal action of $G$ on $X \times X \times X \setminus \Delta$ is properly discontinuous, where $\Delta$ is the set of triples of points of $X$ which are not all distinct. See, for instance, \cite{Tukia94}. If the diagonal action on the set of distinct triples is also cocompact, then $G$ is called a uniform convergence group. 
\end{rmk} 

\begin{rmk} When $X$ is $S^1$, it is easy to see that the uniform convergence in the definition of a convergence group can be replaced by the pointwise convergence.  However, it is important not to conflate the notions of ``uniform convergence group" and ``uniform convergence."
\end{rmk} 

\begin{rmk} When $X$ is $S^2$, the analogue of Convergence Group
  Theorem is not true, i.e., not every discrete convergence group action on $S^2$ comes from a Kleinian group. For instance, one can just start with a Fuchsian representation of a surface group into $\PC$, and quotient the lower hemisphere to a single point. On the other hand, it is a famous open problem whether or not all uniform convergence group actions on $S^2$ come from Kleinian groups.  
\end{rmk}

Finally, the idea of a rainbow, first described in \cite{BaikFuchsian}, will be useful throughout this paper:

\begin{defn}
A lamination $\Lambda$ is said to have a \emph{rainbow} at a point $p \in S^1$ if there is a sequence of leaves $(l_i) = ((a_i, b_i))$ of $\Lambda$ such that $(a_i)$ and $(b_i)$ converge to $p$ from opposite sides.  Such a sequence $(l_i)$ is called a rainbow at $p$ in $\Lambda$. 
\end{defn}

A rainbow is a particularly nice way of approximating a point $p\in S^1$ by leaves of a lamination.  Clearly endpoints of leaves do not admit rainbows.  On the other hand, an observation we shall use later is that for a very full lamination $\Lambda$, these approximations exist for every point that is not an endpoint of a leaf:

\begin{lem}[Rainbow Lemma, Theorem 5.3 of \cite{BaikFuchsian}]
Let $\Lambda$ be a very full lamination of $S^1$.  Every point $p \in S^1$ is either an endpoint of a leaf of $\Lambda$, or there is a rainbow in $\Lambda$ at $p$.  These two possibilities are mutually exclusive. \qed
\end{lem}

\subsection{Semi-conjugacy destroys pseudo-fiberedness}
\label{subsec:semiconjugacy}
We remark in this subsection that semi-conjugacy appears to be irrelevant to the study of pseudo-fibered groups.  This is to be expected in the context of our promotion of pseudo-fibering conjecture, since semi-conjugacy does not preserve convergence actions.  In particular, our Theorem \ref{thm:main1} should be seen as distinct from Margulis's Tits alternative for minimal subgroups of $\Homeop(S^1)$, a point we explain now.

A continuous surjective map $f : S^1 \to S^1 $ is said to be \textbf{monotone} if the preimage of each point is connected. For a group $G$, two actions $\rho$ and $\mu: G \to \Homeop(S^1)$ are said to be semi-conjugate (or $\rho$ is semi-conjugate to $\mu$) if there exists a monotone map $f$ such that $f \circ \rho = \mu \circ f$. For many aspects of the theory of groups of circle homeomorphisms, it is enough to consider the actions up to semi-conjugacy. 

A classical theorem of Poincar\'{e} says that every subgroup of
$\Homeop(S^1)$ either has a finite orbit or is semiconjugate to a
minimal action (meaning every orbit is dense).  Furthermore, a theorem
of Margulis says that subgroups of $\Homeop(S^1)$ that act minimally
either contain $F_2$ as a subgroup or are abelian \cite{Margulis}.
(For more details regarding both of these results, as well as a
general introduction to group actions on $S^1$, see \cite{Ghys01}.)

What we have shown is that for a pseudo-fibered group, even if either
there exists a finite orbit or the action is non-minimal, the Tits
alternative always holds. Hence, the scope of Theorem \ref{thm:main1}
is distinct from the Tits alternative of Margulis. 

However, it is still an interesting question to ask if the pseudo-fibered
groups can be studied up to semi-conjugacy. We observe that 
semi-conjugacy may destroy a pseudo-fibered triple, since the
laminations do not behave well under semi-conjugacy. 
More precisely, we show: 

\begin{prop}
\label{prop:monotone}
Let $f: S^1 \to S^1$ be a monotone map which is not a homeomorphism, and $(\Lambda_1, \Lambda_2)$ a pair of laminations. At most one of the pairs,  $(\Lambda_1, \Lambda_2)$ and $(f(\Lambda_1), f(\Lambda_2))$, can be a pair of very full loose laminations with disjoint endpoint sets. 
\end{prop}

\begin{proof} Since $f$ is not injective, there is a point $p \in S^1$ such that $I := f^{-1}(p)$ has non-empty interior. Let $\hat{p}$ be an endpoint of $I$.  Recall that the Rainbow Lemma says that for each $p \in S^1$ and a very full lamination $\Lambda$, either $p$ is an endpoint of a leaf or there is a rainbow at $p$ in $\Lambda$. 

Suppose $(\Lambda_1, \Lambda_2)$ is a pair of very full loose laminations with disjoint endpoint sets. In particular, there must be a rainbow at $\hat{p}$ in $\Lambda_i$ for at least one of the $i=1$ or $2$. But the image of a rainbow at $\hat{p}$ under $f$ is an infinite set of leaves of $f(\Lambda_i)$ which share a common endpoint. Hence, $f(\Lambda_i)$ cannot be loose. 

For the other direction, suppose $(f(\Lambda_1), f(\Lambda_2))$ is a pair of very full loose laminations with disjoint endpoint sets. From the above argument, we know that there is no rainbow at $\hat{p}$ in $\Lambda_i$ for each $i$. But this means $\hat{p}$ is an endpoint of some leaf in both $\Lambda_1$ and $\Lambda_2$, hence they cannot have disjoint endpoint sets. 
\end{proof} 

There are examples of pseudo-fibered triples whose actions are not
minimal.  Indeed, it is easy to construct examples of pseudo-fibered
groups with finite orbits, and in the next section, we construct
examples of pseudo-fibered groups whose actions are neither minimal
nor have finite orbits. 
It would be interesting to more thoroughly unravel the relationship
between (non)minimal actions, finite orbits, and Conjecture
\ref{conj:promotion}. From this perspective, it is natural to ask
that: 
\begin{ques} Is a pseudo-fibereing semi-conjugacy invariant? Namely,
  for two semi-conjugate actions $\rho_1, \rho_2$ of a group $G$ on
  $S^1$, if $\rho_1$ is pseudo-fibered with two laminations
  $\Lambda_1, \Lambda_2$, is $\rho_2$ also pseudo-fibered with respect
  to a different pair of laminations $\Gamma_1, \Gamma_2$? 
\end{ques} 

Note that even if the above question has an affirmative answer, the
pairs $(\Lambda_1, \Lambda_2)$ and $(\Gamma_1, \Gamma_2)$ may not
be related in any obvious way as we saw in Proposition
\ref{prop:monotone}. 

\section{Free products, torsion, and promotion of pseudo-fibering}\label{sec:promotion}

Theorems \ref{thm:main1} and \ref{thm:main2} can be seen also as partial evidence
for Conjecture \ref{conj:promotion}. In \cite{BaikFuchsian}, a
conjecture similar to Conjecture \ref{conj:promotion} was made without
free indecomposability assumption. The following theorem shows
why Conjecture \ref{conj:promotion}. was amended. 

\begin{thm}
\label{thm:free-product} 
Let $G, H$ be any finite cyclic groups. Then $G*H$ embeds into $\Homeop(S^1)$ as a pseudo-fibered group.
\end{thm}

\begin{proof}
This construction is adapted from a construction in \cite{BS15} that yields a faithful action of any free product of subgroups of $\Homeop(S^1)$ on a new circle, which blows down onto each of the original circles.  We content ourselves with a brief review of the ideas of \cite{BS15}, and a description of how to additionally construct invariant laminations.

\begin{figure}
\centering
\begin{minipage}{.35\textwidth}
  \centering
  \includegraphics[width=0.9\linewidth]{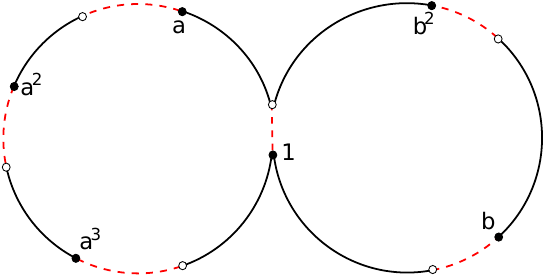}
  \captionof{figure}{The seed graph $\Gamma_0$.}
  \label{fig:seed}
\end{minipage}%
\begin{minipage}{.65\textwidth}
  \centering
  \includegraphics[width=.9\linewidth]{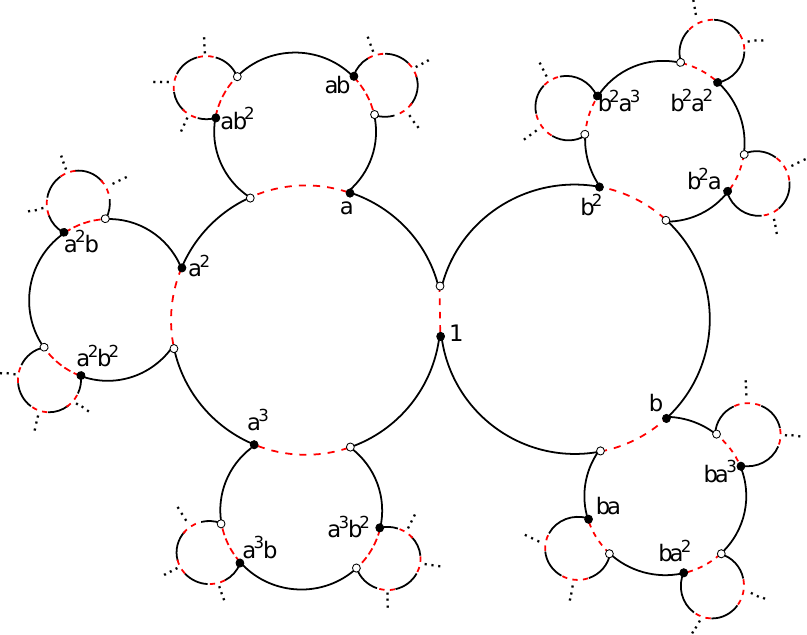}
  \captionof{figure}{The limiting circle $S^1=\overline{\Gamma_\infty}$ laminated by $\Lambda_0$, which is represented by the dashed red lines.}
  \label{fig:limit}
\end{minipage}
\end{figure}

To construct an action of $G * H$ on $S^1$, begin by forming two pointed copies of $S^1$ called $S^1_G$ and $S^1_H$, with marked points both denoted $1$.  Let $G$ act on $S^1_G$ as a finite rotation subgroup, and $H$ act on $S^1_H$ as a finite rotation subgroup.  The $G$-orbit of the point marked 1 in $S^1_G$ is now a copy of $G$, and the $H$-orbit of 1 in $S^1_H$ is a copy of $H$.   Mark all of these points accordingly, wedge $S^1_G$ and $S^1_H$ together at the points marked by 1, blow up all of the marked points, and consistenly label one of the endpoints of the blow-up intervals.  The resulting ``seed", called $\Gamma_0$, is in Figure \ref{fig:seed}, where we have $G=\langle a \mid a^4 = 1\rangle$ and $H = \langle b \mid b^3 = 1 \rangle$.  Now generate an infinite graph $\Gamma_\infty'$ on which $G*H$ acts faithfully.  As in Figure \ref{fig:limit}, write $\Gamma_\infty' = \Lambda_0 \cup \Gamma_\infty$ where $\Lambda_0$ is the orbit of the blown-up intervals in $\Gamma_0$, and $\Gamma_\infty$ is everything else.  The order completion $\overline{\Gamma_\infty}$ is $S^1$, and $\Lambda_0$ is a discrete lamination on this circle.  This proves

\begin{lem}
Let $G, H$ be any finite cyclic groups. Then there exists an injective homomorphism $\rho: G * H \to \Homeop(S^1)$ such that $\rho(G * H)$ admits a discrete invariant lamination. \qed
\end{lem} 

Now one can easily add more leaves to $\Lambda_0$ to construct a $G*H$-invariant very full and loose lamination $\Lambda_1$.  For example, in the left circle of the seed $\Gamma_0$, first take a polygon which has one vertex in each connected component of the complement of the dotted segments and is invariant under the action of $G$.  Then in the region between this polygon and an element of $\Lambda_0$ in $\Gamma_0$, add infinitely many triangles to make the lamination very full and loose in that region. Now fill the other such regions so that lamination becomes $G$-invariant.  One can do the same thing for the right circle to get an $H$-invariant very full loose lamination, and then extend it as a $G*H$-invariant lamination $\Lambda_1$ which contains $\Lambda_0$ as a sublamination.  The final result is in Figure \ref{fig:lambda1}.  $\Lambda_1$ is obviously very full, and loose away from $\Lambda_0$.  It is loose at $\Lambda_0$ because the leaves of $\Lambda_0$ are not contained in gaps---they are instead limits of gaps.

\begin{figure}
\includegraphics{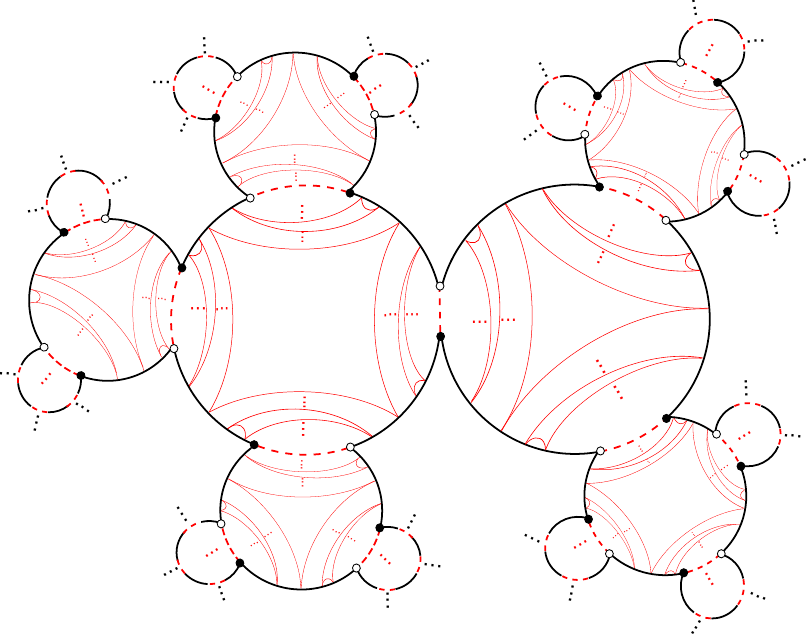}
\caption{The very full and loose lamination $\Lambda_1\supset \Lambda_0$.  We have removed the markings to avoid clutter.}
\label{fig:lambda1}
\end{figure}

We need to construct another $G*H$-invariant very full loose lamination $\Lambda_2$, so that $\Lambda_1$ and $\Lambda_2$ have distinct endpoints.  To build $\Lambda_2$, we first replace each leaf of $\Lambda_0$ with endpoints by four leaves forming, say, a rectangle such that each endpoint of the original dotted segment lies between two adjacent vertices of the rectangle.  In the two regions between the rectangle and the endpoints of the original leaf in $\Lambda_0$, put infinitely many triangles to make the lamination very full and loose.  These choices can obviously be made so that the endpoints of the new leaves are disjoint from $\Lambda_1$, and by working in one region at a time, we can do the construction $G*H$-invariantly, resulting in Figure \ref{fig:rectangles}.  In the regions where all the rectangles are visible, we do the exactly same thing as when construction $\Gamma_1$ from $\Gamma_0$: take a big invariant polygon, and fill out all the complementary regions. The result is shown in Figure \ref{fig:lambda2}.

\begin{figure}
\centering
\begin{minipage}{.3\textwidth}
  \centering
  \includegraphics[width=.6\linewidth]{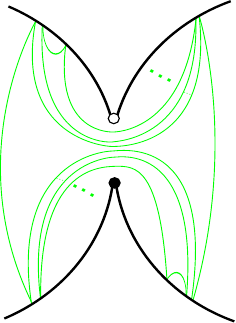}
  \captionof{figure}{Replacing $\Lambda_0$ with rectangles and triangles.}
  \label{fig:rectangles}
\end{minipage}%
\begin{minipage}{.8\textwidth}
  \centering
  \includegraphics[width=.7\linewidth]{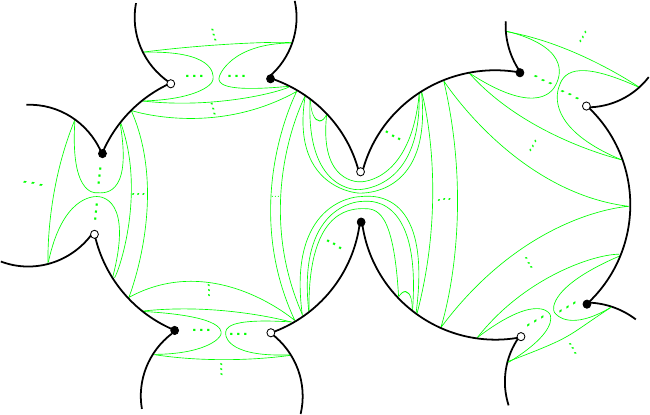}
  \captionof{figure}{The very full and loose lamination $\Lambda_2$.}
  \label{fig:lambda2}
\end{minipage}
\end{figure}

Finally, to show $(\rho(G*H),\Lambda_1,\Lambda_2)$ is a pseudo-fibered triple, we need to show that every element of $G*H = \rho(G*H)$ has countably many fixed points in its action on $S^1$.  There are two ways to prove this, either using Bass-Serre theory, or the existence of even more $G*H$-invariant laminations under this action.  

For the latter approach, we quote the following result:

\begin{thm}[\cite{BaikFuchsian}] 
\label{thm:mobius}
Every subgroup of $\Homeop(S^1)$ admitting three very full invariant laminations with distinct endpoint sets is M\"obius-like, meaning every element is (individually) M\"{o}bius-like. \qed
\end{thm} 

Since both $G$ and $H$ are finite, there is a large freedom to construct laminations inductively as before.  Indeed, we can slightly perturb the construction of $\Lambda_2$ to get a third very full and loose invariant lamination $\Lambda_3$ with endpoints distinct from $\Lambda_1$ and $\Lambda_2$.  Theorem \ref{thm:mobius} then implies every element of $G*H$ is M\"{o}bius-like, hence has finitely many fixed points.

Alternatively, to show every element of $G*H$ acts on $S^1$ with at most two fixed points, we can use Bass-Serre theory.  Clearly torsion elements of $G*H$ act freely on $S^1$.  Nontorsion elements must have their fixed points in the subset $S^1 \setminus \Gamma_\infty$, which can be identified with the ends of the Bass-Serre tree for $G*H$.  Standard results now imply such elements have two fixed points in $S^1$.
\end{proof}


Theorem \ref{thm:free-product} has an immediate corollary in the context of 3-manifold groups.

\begin{cor} \label{thm:consum-lensspace}
Let $M$ be a connected sum of two lens spaces. Then $\pi_1(M)$ admits a pseudo-fibered group action on $S^1$. 
\end{cor} 


\section{The Tits Alternative for Pseudo-Fibered Groups}
\label{sec:Tits}

Throughout this section, $G$ will denote a torsion-free pseudo-fibered group, except where explicitly indicated elsewhere.  This will allow us to apply Theorem \ref{thm:pAlike_elements}, which we may do sometimes without mentioning.

\subsection{Proof of Theorem \ref{thm:dynamics}}
\label{subsec:dynamics}
To prove that pseudo-fibered groups satisfy the Tits alternative, we first prove Theorem \ref{thm:dynamics}, which we recall says that two elements of a pseudo-fibered group have either equal or disjoint sets of periodic points.  This can be done, for instance, by analyzing how each element of the group acts on the quotient of the circle obtained by collapsing leaves of an invariant lamination. One can show that such a quotient is a dendrite as in \cite{Winkel16}, and this point of view has its own advantages. But for our purpose, it is simpler to analyze the group action on the circle directly. 

We will need a number of lemmas.  The first follows immediately from the definitions, so we leave its proof to the reader:

\begin{lem} 
\label{lem:two_dense_imply_totally_disconnected}
Suppose there are two dense laminations with distinct endpoints. Then each of the laminations is totally disconnected. \qed
\end{lem} 

Recall that by definition, a lamination $\Lambda$ is a type of closed subset of $(S^1 \times S^1 - \Delta)/(x,y)\sim(y,x)$.  Thus it makes sense to talk about the neighborhood in $\Lambda$ of a leaf, isolated leaves, \emph{etc}.

\begin{lem} 
\label{lem:totally_disconnected_very_full_imply_boundary_or_limit}
Each leaf of a totally disconnected very full lamination is either a boundary leaf of a gap, or is the limit of an infinite sequence of gaps. 
\end{lem}

\begin{proof}
This is a direct consequence of the definition of a very full lamination.  Indeed, if a neighbourhood of a leaf meets no gaps, it must be foliated, contradicting that the lamination is totally disconnected.
\end{proof} 

\begin{lem} 
\label{lem:characterizaion_of_loose}
Let $\Lambda$ be a totally disconnected very full lamination. Then $\Lambda$ is loose if and only if the following conditions are satisfied:
\begin{itemize}
\item[(1)] For each $p \in S^1$, at most finitely many leaves of $\Lambda$ have $p$ as an endpoint. 
\item[(2)] There are no isolated leaves. 
\end{itemize} 
\end{lem}

\begin{proof} 
Suppose $\Lambda$ is loose. Then each $p\in S^1$ is an endpoint of at most two leaves (maybe none) of $\Lambda$. Hence condition (1) follows immediately.

For condition (2), suppose there exists a leaf $L$ which is isolated. Let $J_1$ and $J_2$ be the connected components of the complements of the endpoints of $L$ in $S^1$.  The fact that $L$ is isolated means there exists an open arc $I_1$ containing one endpoint of $L$, and another open arc $I_2$ containing the other endpoint of $L$, such that for each $i = 1, 2$, there exists no leaf connecting $I_1 \cap J_i$ to $I_2 \cap J_i$.  For $i=1,2$, define $\Lambda_i$ to be the set of leaves of $\Lambda$ with endpoints in $J_i$ that are visible from both endpoints of $L$. Both $\Lambda_i$ are nonempty since $\Lambda$ is dense. For each $i=1,2$, $\Lambda_i\cup\{L\}$ is the set of boundary leaves of a gap $P_i$ of $\Lambda$. This contradicts looseness of $\Lambda$, since both $P_1$ and $P_2$ are gaps sharing some of their vertices. 

Now for the converse, assume $\Lambda$ satisfies the conditions (1) and (2). Suppose $p$ is a common endpoint of two gaps. Let $L_1, L_2$ be the innermost leaves ending at $p$. (It's possible $L_1=L_2$, but this does not change what follows.) Then $L_1$ is either isolated (absurd by condition (2)), or is approximated by infinitely many leaves. Since these leaves cannot cross $L_2$, they must end at $p$, contradicting condition (1).
\end{proof} 


\begin{lem}
\label{lem:hyperbolic_elements} Let $G$ be a pseudo-fibered group, and
let $\Lambda$ be a very full loose $G$-invariant lamination. If $g$ is
a hyperbolic element, then the fixed points of $g$ in $S^1$ are not an
endpoint of any leaf of $\Lambda$. 
\end{lem}

\begin{proof}
Let $p, q$ be the fixed points of $g$. Suppose $p$ is an endpoint of a leaf $l$. 

If the other endpoints of $l$ is not $q$, then the set $\{g^{\circ n}(l) : n \in \ZZ\}$ gives an infinite set of leaves each of which has $p$ as an endpoint, contradicting Lemma \ref{lem:characterizaion_of_loose}. Hence, we may assume the other endpoint of $l$ is $q$. 

Let $I$ be a connected component of $S^1 \setminus \{p, q\}$. Define $\Lambda_I$ to be the set of leaves whose endpoints are in $I$, and visible from both $p$ and $q$. Assume $\Lambda_I$ is nonempty. Since $\{L\} \cup \Lambda_I$ bound a gap, say $P$, $\Lambda_I$ must be finite. But this means there must be a leaf in $\Lambda_I$ which connects $p$ to a point in $I$ which we already saw impossible. Now assume $\Lambda_I$ is empty. This means there exists a family of infinitely many leaves contained in $I$ which accumulate to $l$. But since $g$ acts as a translation on $I$, this is impossible (if $l'$ is a leaf close enough to $l$, then $l'$ and $gl'$ must be linked). 
\end{proof} 

\begin{thm}[Solodov \cite{Navasbook}]
\label{thm:Navas}
If $G$ is a subgroup of $\Homeop(\RR)$ such that each non-trivial element has at most one fixed point, and there is no global fixed point, 
then $[G,G]-\{Id\}$ consists of fixed-point-free elements.
\end{thm} 
\begin{proof} See Step 4 in the proof of Theorem 2.2.36 in \cite{Navasbook}. 
\end{proof} 

\begin{lem} 
\label{lem:pAlike_share_fixedpoints}
Let $(G, \Lambda_1, \Lambda_2)$ be a pseudo-fibered triple. Suppose $g \in G$ is properly pseudo-Anosov-like. If $h \in G$ shares a fixed point $p$ with $g$, then $\Fix_h = \Fix_g$. In particular, $h$ is also a properly p-A-like element. 
\end{lem} 
\begin{proof} First, we show that $h$ cannot be hyperbolic. Since $p$ is a fixed point of $g$, then by Theorem \ref{thm:pAlike_elements} one of $\Lambda_1$ or $\Lambda_2$ contains a gap which has $p$ as a vertex. But then Lemma \ref{lem:hyperbolic_elements} says that $p$ cannot be a fixed point of an hyperbolic element. Therefore, by Theorem \ref{thm:pAlike_elements}, $h$ must be pseudo-Anosov-like. 

Without loss of generality, we may assume that $p$ is an attracting fixed point of both $g, h$, and $\Lambda_1$ contains the attracting polygon $P_g$ of $g$. Let $q$ be a vertex of $P_g$ which is connected to $p$ by a boundary leaf $l$ of $P_g$. If $q$ is not a fixed point of $h$, then $\{h^{\circ n}(l) : n \in \ZZ\}$ is an infinite set of leaves which share $p$ as a common endpoint. This is impossible by Lemma \ref{lem:characterizaion_of_loose} since $\Lambda_1$ is loose. This inductively shows that all vertices (i.e., all attracting fixed points of $g$) are fixed by $h$.  Theorem \ref{thm:pAlike_elements} says there is no leaf connecting an attracting fixed point to a repelling fixed point for a given p-A-like elements. Hence, all attracting fixed points of elements of $g$ are attracting fixed points of $h$. Applying the same argument to the attracting polygon of $h$, one concludes that $P_g$ is in fact the attracting polygon of $h$ as well. 

We showed the attracting polygon of $g$ and the attracting polygon of $h$ must coincide if $p$ is attracting. How about the repelling polygons? Note Theorem \ref{thm:pAlike_elements}(2) implies both the repelling polygon of $g$ and the repelling polygon of $h$ are contained in $\Lambda_2$. 

Case 1. Suppose $g$ (hence, also $h$) has at least three attracting fixed points. Since both $g$ and $h$ has a unique repelling fixed point between two adjacent vertices of $P_g$, the only way for the repelling polygons to be unlinked is that they coincide. Therefore, $\Fix_g = \Fix_h$. 

Case 2. Now suppose $g$ has only two attracting fixed points. We know that each connected components of $S^1 \setminus \Fix_g$ has exactly one repelling fixed point of $g$ and exactly one repelling fixed point of $h$. Let $I$ be a connected component of $S^1 \setminus \Fix_g$, $r_g$ the repelling fixed point of $g$ on $I$, $r_h$ the repelling fixed point of $h$ on $I$.  Let $H$ be the subgroup of $G$ generated by $g, h$. First note that every element of $H$ fixes the endpoints of $I$. From the previous arguments, we know that each element of $H$ is p-A-like and has exactly two attracting fixed points and two repelling fixed points. This implies that each element of $H$ has exactly one fixed point in $I$. Now we apply Theorem \ref{thm:Navas} by identifying $I$ with $\RR$. If $r_g \neq r_h$, $H$ has no global fixed point in $I$, and the theorem says that $H$ is abelian. But if $g$ and $h$ commute, then $r_g$ and $r_h$  coincide, a contradiction. Therefore, $r_g$ and $r_h$ must coincide from the beginning, and $\Fix_g = \Fix_h$ again. 
\end{proof} 

The following generalizes Lemma \ref{lem:pAlike_share_fixedpoints} for p-A-like elements that could be elliptic. We use $\Per_f$ to denote the set of all periodic points of a homeomorphism $f$.

\begin{lem} \label{lem:pA-like-share-periodicpoints} Let $g$ be a p-A-like element in a pseudo fibered group $G$. If $h\in G$ shares a periodic point $p$ with $g$, then $\Per_h = \Per_g$ and $h$ is also p-A-like.
\end{lem}

\begin{proof} Take powers $g^n$ and $h^m$ such that $g^n$ is properly p-A-like and $p$ is fixed for $h^m$. Note that every periodic point of a properly p-A-like element is fixed. So $\Fix_{g^n} = \Per_{g^n} = \Per_g$. Applying Lemma \ref{lem:pAlike_share_fixedpoints}, we get that $\Fix_{g^n}=\Fix_{h^m}$ and that $h^m$ is properly p-A-like. So $h$ is p-A-like and we also have $\Per_h = \Fix_{h^m}$, that agrees with $\Fix_{g^m} = \Per_g$.    
\end{proof}

The previous lemma establishes ``half" of the dynamical alternative.  The other half follows from the next lemma.

\begin{lem} 
\label{lem:purely_hyperbolic_subgroup} 
Let $G$ be a pseudo-fibered group. Suppose $g, h$ are hyperbolic elements of $G$ which share a fixed point $p$. Then every element of the subgroup generated by $g, h$ has the same fixed points as $g$. 
\end{lem} 
\begin{proof} Since $p$ is fixed by $g, h, g^{-1}, h^{-1}$, any element of the subgroup $H$ of $G$ generated by $g, h$ fixes $p$. But Lemma \ref{lem:pAlike_share_fixedpoints} says no p-A-like element shares a fixed point of a hyperbolic element. Hence, all elements of such a subgroup must be hyperbolic. 

Now we apply Theorem \ref{thm:Navas} to $H$ by identifying $S^1 \setminus \{p\}$ with $\RR$. Just as in Case 2 of the proof of Lemma \ref{lem:pAlike_share_fixedpoints}, we conclude that $H$ has a (unique) global fixed point in $S^1 \setminus \{p\}$. 
\end{proof} 

Combining Lemmas \ref{lem:pA-like-share-periodicpoints} and \ref{lem:purely_hyperbolic_subgroup} with Theorem \ref{thm:pAlike_elements}, we immediately conclude Theorem \ref{thm:dynamics}. \qed

\subsection{Proof of Theorem \ref{thm:main1}}
\label{subsec:Tits}
We now combine Theorem \ref{thm:dynamics} with two known results to prove Theorem \ref{thm:main1}.  The first result is a very well-known tool in geometric group theory (for instance, see Ch. II.B of \cite{HarpeGGT}):

\begin{thm}[Ping-pong lemma] 
\label{thm:pingpong} 
Let $G$ be a group acting on a set $X$. Let $g_1$, $g_2$ be elements of $G$. Suppose there exist disjoint nonempty subsets
 $X_1^+, X_1^-, X_2^+, X_2^-$ of $X$ such that 
$g_i(X - X_i^-) \subset X_i^+, g_i^{-1}(X - X_i^+) \subset X_i^-$ for each $i = 1, 2$.  
Then the subgroup generated by $g_1, g_2$ is free. \qed 
\end{thm}

A proof of the second result we need can be found in many places, e.g.\ \cite{Ghys01} or \cite{Navasbook}:

\begin{thm}[H\"older]
Let $K$ be a subgroup of $\Homeop(\RR)$ which acts freely on $\RR$.  Then $K$ is abelian. \qed
\end{thm}

Now the ping-pong lemma implies the first case of the Tits alternative:

\begin{lem} \label{lem:disjoint-periodics} Let $G$ be a pseudo-fibered group and $g_1,g_2\in G$. If $\Per_{g_1}$ and $\Per_{g_2}$ are disjoint, then there are powers of $g_1$ and $g_2$ that generate a non-abelian free subgroup of $G$.
\end{lem}

\begin{proof} We can replace $g_1$ and $g_2$ by some powers that satisfy $\Per_{g_i} = \Fix_{g_i}$. (If $g_i$ is hyperbolic, or properly p-A-like, no power needs to be taken. If $g_i$ is elliptic p-A-like, take a power that is properly p-A-like). Take $X_i^+$ to be a neighborhood of the attracting fixed points of $g_i$ and $X_i^-$ to be a neighborhood of the repelling fixed points of $g_i$ (for $i=1,2$). Since $\Fix_{g_1}$ and $\Fix_{g_2}$ are disjoint by hypothesis, we can take $X_i^+$, $X_i^-$ (for $i=1,2$) to be all disjoint.
Now let $h_i$ be a high enough power of $g_i$ so that $h_i (S^1 - X_i^-) \subset X_i^+$ and $h_i^{-1}(X - X_i^+) \subset X_i^-$, for $i=1,2$. This is possible because of the dynamics of hyperbolic and properly p-A-like elements. By the Ping-Pong lemma, $h_1$ and $h_2$ generate a free subgroup.  
\end{proof}

H\"older's theorem implies the alternative case:

\begin{lem} \label{lem:same-periodics} Let $G$ be a pseudo-fibered group, $g\in G$ and $P = \Per_g$. Consider the subgroup $H = \{ h\in G : hP = P\}$ of $G$ that leaves $P$ invariant. Then:
\begin{enumerate}
\item $H = \{h\in G : \Per_h = P \}$
\item $H$ is virtually abelian.
\end{enumerate}
\end{lem}

\begin{proof}
Let $h\in H$. Since $hP=P$ and $P=\Per_g$ is finite, then $P\subset \Per_h$. Then by Theorem \ref{thm:dynamics} we get that $\Per_h = \Per_g =P$. This proves the first assertion, the other inclusion being trivial.

For the second assertion, let $K \unlhd H$ consist of elements that stabilize $\Per_g$ pointwise.  That is,
$$K = \{ h \in H \mid \Fix_h = \Per_g \}.$$
Note
$$[H:K] \leq \frac{|\Per_g|}{2},$$
so in particular, to show $H$ is virtually abelian, it suffices to show $K$ is abelian.  Assertion (1) implies $K$ acts freely on each component of $S^1 \setminus \Per_g$, each of which is an interval.  So by H\"{o}lder's theorem, $K$ is abelian.
\end{proof}

Theorem \ref{thm:main1} now follows from Theorem \ref{thm:dynamics}, since if $H$ is a subgroup of $G$, then either there exist two elements of $H$ with distinct periodic point sets, or $H$ has a global set of periodic points.  In the first case, apply Lemma \ref{lem:disjoint-periodics}; in the second case, apply Lemma \ref{lem:same-periodics}.
\qed

\subsection{Remarks on the proofs of Theorems \ref{thm:main1} and \ref{thm:dynamics}}\label{ss:remarks}
Conjecturally, non-elementary pseudo-fibered groups are word-hyperbolic. For word-hyperbolic groups, a stronger version of the Tits alternative holds: an infinite subgroup of a word-hyperbolic group either contains a free group of rank 2 or is virtually \emph{cyclic}. To obtain this stronger Tits alternative, one needs to strengthen Lemma \ref{lem:same-periodics}. Let $H$ be a subgroup as in Lemma \ref{lem:same-periodics}, and assume it is actually abelian. Then Ghys \cite{Ghys01} provides an $H$-invariant measure on each connected component of $S^1-\Per_H$. The problem is that this measure may not have full support. Even when the action of $G$ on $S^1$ is minimal, it is still not clear if it can be shown that we have an invariant measure of full support on the complement of $\Per_H$. The stronger Tits alternative would easily follow from there.

More directly, it is easy to verify that this stronger Tits alternative is equivalent to showing that the subgroup $K$ in the proof of Lemma 4.12 is isomorphic to $\mathbb{Z}$.  A stronger version of H\"{o}lder's theorem says that for any abelian subgroup of $\Homeop(\RR)$, there is a blowdown of $\RR$ such that the induced action of the subgroup is faithful and by translations.  Thus, showing $K$ is $\mathbb{Z}$ is equivalent to showing that this translation action is not minimal.

Finally, we remark that if we knew that pseudo-Anosov-like elements really were pseudo-Anosov, the stronger Tits alternative would follow from the fact that two pseudo-Anosovs commute if and only if they are powers of some other pseudo-Anosov. This can be seen by considering the action of the mapping class group on Thurston's compactification of Teichm\"{u}ller space: if two pseudo-Anosovs commute, they must fix the same axis. The mapping class group acts discretely on Teichm\"{u}ller space, hence, the subgroup generated by the two pseudo-Anosovs must act discretely on the axis in Teichm\"{u}ller space. Each pseudo-Anosov acts by translations on this axis, so we conclude that the subgroup the two generate is isomorphic to $\mathbb{Z}$.

\section{Purely hyperbolic Pseudo-fibered groups and convergence group actions on $S^2$}
\label{sec:purelyhyperbolic}
We now restrict our attention to the special class of purely
hyperbolic pseudo-fibered groups.  In \cite{BaikFuchsian}, it was
conjectured that such groups are always Fuchsian, or, equivalently,
convergence subgroups of $\Homeop(S^1)$. While this conjecture remains
open, Theorem \ref{thm:main2} shows that, as expected, purely
hyperbolic pseudo-fibered groups act on the 2-sphere as convergence
groups.  We remark that a similar idea of relating group action on $S^1$ with a
geometric origin to convergence group action on $S^2$ has been carried
out in the context of pseudo-Anosov flows in Fenley's program (see
Section 4 of \cite{Fenley12}). 


To prove Theorem \ref{thm:main2}, we begin by reviewing the construction of \cite{BaikFuchsian}, which is responsible for the existence of an $S^2$ on which a pseudo-fibered group can act, and inspired by results of Cannon and Thurston \cite{CannonThurston}.  Moore's theorem \cite{Moore25} implies that for any pseudo-fibered triple $(G, \Lambda_1, \Lambda_2)$, there exists a quotient map $\pi: S^1 \to S^2$, constructed by first identifying two disks laminated by $\Lambda_1$ and $\Lambda_2$ (respectively) along their common boundary $S^1$, and then collapsing all the gaps of the $\Lambda_i$ to points.  Since each lamination is $G$-invariant, this induces a $G$-action on $S^2$ such that $\pi$ is $G$-equivariant.  We call this map $\pi$ the \emph{Cannon-Thurston map} for the pseudofibered triple $(G, \Lambda_1,\Lambda_2)$.  For details, one can also consult Section 14 of \cite{CannonThurston}.  A basic observation about this construction is

\begin{lem} 
\label{lem:lifting_convergent_seq} 
 Let $(G, \Lambda_1, \Lambda_2)$ be a pseudo-fibered triple. Suppose there exists a sequence $(\x_i)$ of points in $S^2 (=\pi(S^1))$ which converges to $\x$, and a sequence $(g_i)$ of elements of $G$ such that $g_i(\x_i)$ converges to $\x'$ in $S^2$. Then, passing to subsequences if necessary, there exists a sequence $(x_i)$ of points in $S^1$ converging to $x$ such that $g_i(x_i)$ converges to $x'$ in $S^1$, where $\x_i= \pi(x_i)$ and $\x' = \pi(x')$.
\end{lem}

\begin{proof}
This is straightforward because $S^1$ is compact, and $\pi$ is continuous, surjective and $G$-equivariant. 
\end{proof} 

We now state a few dynamical lemmas, after which we will prove Theorem \ref{thm:main2}.

\begin{lem}
\label{lem:nested_neighborhood} 
Let $G$ be a group acting on $S^1$ such that there exists a $G$-invariant lamination with a rainbow at $p \in S^1$.  Suppose there exists a sequence $(g_i)$ of elements of $G$ such that for any neighborhood $U$ of $p$, $g_i(U)$ intersects $U$ nontrivially for all large $i$. Then $p$ is an accumulation point of some fixed points of the elements in the sequence $(g_i)$. 
\end{lem}

\begin{proof}
See the proof of Proposition 7.5 in \cite{BaikFuchsian}. 
\end{proof}

\begin{lem}
\label{lem:accumulation_point}
 Let $(G, \Lambda_1, \Lambda_2)$ be a pseudo-fibered triple. 
 Suppose $(x_i)$ is a sequence of points in $S^1$ which converges to $x$, and there exists a sequence $(g_i)$ of elements of $G$ such that $g_i(x_i)$ converges to $x'$. Then either $x$ is an accumulation points of fixed points of the sequence $(g_{i+1}^{-1} \circ g_i)$ or $x'$ is an accumulation point of fixed points of the sequence $(g_{i+1} \circ g_i^{-1})$. Moreover, if $x_i = x$ for all $i$, then $x'$ must be an accumulation point of fixed points of (any subsequence of) the sequence $(g_{i+1} \circ g_{i}^{-1})$. 
\end{lem}

\begin{proof} 
This is a straightforward consequence of Lemma \ref{lem:nested_neighborhood}. See, for instance, the proof of Proposition 7.6 in \cite{BaikFuchsian}.
\end{proof} 

\begin{proof}[Proof of Theorem \ref{thm:main2}]
Suppose the negation of the conclusion. Then there exists an infinite sequence of distinct elements $(g_i)$ of $G$ which does not have the convergence property, \ie, the set $\{g_i\}$ does not act properly discontinuously on the set of triples of distinct points of $S^2$. More precisely, this means that, after passing to a subsequence of $(g_i)$, there exist three convergent sequences in $S^2$
$$\x_i \to \x, \ \ \y_i \to \y, \ \ \z_i \to \z, \ \ \x_i\neq \y_i \neq \z_i \neq \x_i, \ \ \x \neq \y \neq \z \neq \x,$$
and three elements $\x',\y',\z' \in S^2$ such that
$$g_i(\x_i) \to \x', \ \ g_i(\y_i) \to \y', \ \ g_i(\z_i) \to \z', \ \ \x'\neq \y' \neq \z'\neq \x'.$$
We conclude that there exist sequences and points in $S^1$ such that
$$x_i \to x, \ \ y_i \to y, \ \ z_i \to z, \ \ x_i \neq y_i \neq z_i \neq x_i, \ \ x \neq y \neq z \neq x$$
and three elements $x',y',z' \in S^1$ such that
$$g_i(x_i) \to x', \ \ g_i(y_i) \to y', \ \ g_i(z_i)\to z', \ \ x'\neq y' \neq z' \neq x',$$
where in our notation, $p$ is some fixed point in the preimage of $\overline{p}$ under $\pi$ for any $\overline{p} \in S^2$, in accordance with Lemma \ref{lem:lifting_convergent_seq}.  In words, we can lift sequences exhibiting the failure of $G$ to act as a convergence group on $S^2$ to sequences exhibiting the failure of $G$ to act as a convergence group on $S^1$.

Since $\Lambda_1$ and $\Lambda_2$ do not share any endpoints, for each $p \in \{x, y, z\}$, there exists a rainbow at $p$ in at least one of the $\Lambda_i$.  In particular, for each $p \in \{x,y,z\}$ there exists a leaf $L_p$ which separates $p$ from the other two points in $\{x, y, z\} \setminus \{p\}$ (which lamination $L_p$ belongs to is not important, and $L_x, L_y, L_z$ are not necessarily leaves of the same lamination).  Passing to a subsequence, we may assume that each of the sequences of pairs of points described by the endpoints of the leaves $(g_i(L_x))$, $(g_i(L_y))$, $(g_i(L_z))$, $(g_i^{-1}(L_x))$, $(g_i^{-1}(L_y))$, $(g_i^{-1}(L_z))$ converges to a pair of points, which are possibly not distinct.

By Lemma \ref{lem:accumulation_point}, either at least two of $x,y,$ and $z$ are accumulation points of fixed points of the sequence $(g_{i+1}^{-1} \circ g_i)$ or at least two of $x',y',$ and $z'$ are accumulation points of fixed points of the sequence $(g_{i+1} \circ g_i^{-1})$. Without loss of generality, we assume that $x'$ and $y'$ are accumulation points of fixed points of the sequence $(g_{i+1} \circ g_i^{-1})$, possibly after exchanging the roles of $g_i$ and $g_i^{-1}$.  Furthermore, since each element of $G$ has exactly two fixed points, there exists a subsequence $(g_{{i_j}+1} \circ g_{i_j}^{-1})$ of the sequence $(g_{i+1} \circ g_i^{-1})$ such that $x'$ and $y'$ are the \emph{only} accumulation points of the fixed points of the $g_{{i_j}+1} \circ g_{i_j}^{-1}$. By the second statement of  Lemma \ref{lem:accumulation_point}, if $p\in S^1$ is such that $g_i(p)$ converges to $p'$, then $p'$ must be either $x'$ or $y'$. In particular, considering the sequence $(g_i(L_z))$, which converges to some pair of possibly nondistinct points $\{e_1, e_2\}$, what we have shown implies each $e_i$ is either $x'$ or $y'$. 

If $e_1 \neq e_2$, since laminations are required to be closed, the limit of the sequence of leaves $(g_i(L_z))$) must be a leaf connecting $x'$ to $y'$ in the lamination $\Lambda_i$ containing $L_z$.  But since $\x'=\pi(x')$ and $\y'=\pi(y')$ are assumed to be distinct, by the definition of the Cannon-Thurston map $\pi$, their preimages cannot be connected by a leaf.  This contradiction implies $e_1 = e_2$.

Let's assume that $e_1 = e_2 = x'$.  We will show that $x'$ is not distinct from both $y'$ and $z'$.  In the case $e_1 = e_2 = y'$, the same argument would lead us to contradict that $y'$ is not distinct from both $x'$ and $z'$.

Let $I_y$ be the closure of the connected component of $S^1 - L_z$ which contains $y$, and define $I_z$ similarly for $z$.  Take a nested sequence of closed neighborhoods $(U_i)$ of $x'$ such that $U_i \to x'$ as $i \to \infty$. Passing to a subsequence, one may assume that the endpoints of $g_i(L_z)$ are contained in $U_i$.
 
For each $i$, there are two possibilities: either $g_i(I_z) \subset U_i$, or $g_i(I_y) \subset U_i$.  Suppose the former happens for infinitely many $i$.  Since $(z_i)$ converges to $z$, then for all large enough $i$, $z_i$ is in $I_z$. Hence, $g_i(z_i) \in U_i$ for infinitely many $i$, so some subsequence of $g_i(z_i)$ converges to $x'$.  But this is impossible since $z'$ is assumed to be distinct from $x'$, and $g_i(z_i)$ converges to $z'$. If instead $g_i(I_y) \subset U_i$ for infinitely many $i$, then we similarly contradict the assumption that $y' \neq x'$.
\end{proof}

We remark that this proof almost goes through to show that a purely
hyperbolic pseudo-fibered group $G$ acts as a convergence group on the
circle $S^1 = \pi^{-1}(S^2)$.  The only gap to consider is the case
where $e_1 \neq e_2$.

We conclude the paper with some questions which naturally arise from the results of this paper. 
\begin{ques} For fibered groups, how do hyperbolic elements and pseudo-Anosov-like elements interact? What can be say about the group structure of a pseudo-fibered group in terms of the dynamical feature of the elements of the group? 
\end{ques} 
\begin{ques} \label{ques:uniformconvergence}
For a fibered group action on $S^1$, without using hyperbolic geometry, can we abstractly show that the induced action on $S^2$ as in this section is a uniform convergence group? 
\end{ques} 

Answering (or rather understanding) above questions appropriately, one might hope to get a characterization of fibered hyperbolic 3-manifold groups via its action on $S^1$ with invariant laminations. 
For instance, assuming Cannon's conjecture \cite{CannonConjecture}, an
affirmative answer to Question \ref{ques:uniformconvergence} together
with a result of Bowditch \cite{Bowditch98} would imply that such a
group is always a closed hyperbolic 3-manifold group. 

\bibliographystyle{abbrv}
\bibliography{biblio}

\begin{thebibliography}{10}

\bibitem{BaikFuchsian}
H.~Baik.
\newblock Fuchsian groups, circularly ordered groups, and dense invariant
  laminations on the circle.
\newblock {\em Geom. Topol.}, 19(4):2081--2115, 2015.

\bibitem{BS15}
H.~Baik and E.~Sampterton.
\newblock Spaces of invariant circular orders of groups.
\newblock to appear in Groups Geom. Dyn.

\bibitem{Bowditch98}
B.~Bowditch.
\newblock A topological characterization of hyperbolic groups.
\newblock {\em J. Amer. Math. Soc.}, 11:643--667, 1998.

\bibitem{Calegarinote}
D.~Calegari.
\newblock Foliations and geometrization of 3-manifolds.
\newblock Lecture note for the course 'Foliations and 3-manifolds' at
  University of Chicago, 2003.

\bibitem{Calebook}
D.~Calegari.
\newblock {\em Foliations and the Geometry of 3-Manifolds}.
\newblock Oxford Science Publications, 2007.

\bibitem{CalDun03}
D.~Calegari and N.~Dunfield.
\newblock Laminations and groups of homeomorphisms of the circle.
\newblock {\em Invent. Math.}, 152:149--207, 2003.

\bibitem{CannonConjecture}
J.~W. Cannon and E.~L. Swenson.
\newblock Recognizing constant curvature discrete groups in dimension 3.
\newblock {\em Trans. Amer. Math. Soc.}, 350(2):809--849, 1998.

\bibitem{CannonThurston}
J.~W. Cannon and W.~P. Thurston.
\newblock Group invariant {P}eano curves.
\newblock {\em Geometry \& Topology}, 11:1315--1355, 2007.

\bibitem{CassonJungreis94}
A.~Casson and D.~Jungreis.
\newblock Convergence groups and $\mbox{S}$eifert fibered 3-manifolds.
\newblock {\em Invent. math.}, 118:441--456, 1994.

\bibitem{casson1988automorphisms}
A.~J. Casson and S.~A. Bleiler.
\newblock {\em Automorphisms of surfaces after Nielsen and Thurston}, volume~9.
\newblock Cambridge University Press, 1988.

\bibitem{HarpeGGT}
P.~de~la Harpe.
\newblock {\em Topics in geometric group theory}.
\newblock Chicago Lectures in Mathematics. University of Chicago Press, 2000.

\bibitem{Fenley12}
S.~Fenley.
\newblock Ideal boundaries of pseudo-anosov flows and uniform convergence
  groups, with connections and applications to large scale geometry.
\newblock {\em Geom. Topol.}, 17:1--110, 2012.

\bibitem{Fenley16}
S.~Fenley.
\newblock Quasigeodesic pseudo-anosov flows in hyperbolic 3-manifolds and
  connections with large scale geometry.
\newblock {\em Advances in Mathematics}, 303:192--278, 2016.

\bibitem{Frankel12}
S.~Frankel.
\newblock Quasigeodesic flows and sphere filling curves.
\newblock arXiv:1210.7050.

\bibitem{Gabai91}
D.~Gabai.
\newblock Convergence groups are $\mbox{F}$uchsian groups.
\newblock {\em Bull. Amer. Math. Soc.}, 25(2):395--402, 1991.

\bibitem{Ghys01}
E.~Ghys.
\newblock Groups acting on the circle.
\newblock {\em L'Enseignement Math\'{e}matique}, 47:329--407, 2001.

\bibitem{Hinkkanen90}
A.~Hinkkanen.
\newblock Abelian and nondiscrete convergence groups on the circle.
\newblock {\em Trans. Amer. Math. Soc.}, 318:87--121, 1990.

\bibitem{Kovacevic99}
N.~Kova{\v c}evi{\'c}.
\newblock Examples of m{\"o}bius-like groups which are not m{\"o}bius groups.
\newblock {\em Trans. Amer. Math. Soc.}, 351(12):4823--4835, 1999.

\bibitem{Mann15}
K.~Mann.
\newblock Spaces of surface group representations.
\newblock {\em Invent. Math.}, 201(2):669--710, 2015.

\bibitem{Margulis}
G.~Margulis.
\newblock Free subgroups of the homeomorphism group of the circle.
\newblock {\em C. R. Acad. Sci. Paris S\'er. I Math.}, 331(9):669--674, 2000.

\bibitem{Moore25}
R.~L. Moore.
\newblock Concerning upper semi-continuous collections of continua.
\newblock {\em Trans. Amer. Math. Soc.}, 27:416--428, 1925.

\bibitem{Navasbook}
A.~Navas.
\newblock {\em Groups of Circle Diffeomorphisms}.
\newblock Chicago Lectures in Mathematics. The University of Chicago Press,
  2011.

\bibitem{BarbotFenley15}
S.~F. Thierry~Barbot.
\newblock Free seifert pieces of pseudo-anosov flows.
\newblock arXiv:1512.06341.

\bibitem{Thurston97}
W.~P. Thurston.
\newblock Three-manifolds, $\mbox{F}$oliations and $\mbox{C}$ircles,
  $\mbox{I}$.
\newblock {\em arXiv:math/9712268v1 [math.GT]}, 1997.

\bibitem{Tukia88}
P.~Tukia.
\newblock Homeomorphic conjugates of $\mbox{F}$uchsian groups.
\newblock {\em J. reine angew. Math.}, 398:1--54, 1988.

\bibitem{Tukia94}
P.~Tukia.
\newblock Convergence groups and {G}romov's metric hyperbolic spaces.
\newblock {\em New Zealnd J. Math.}, 23(2):157--187, 1994.

\bibitem{Winkel16}
E.~T. Winkel.
\newblock Sticky and slippery laminations on the circle.
\newblock Master's thesis, University of Bonn, 2016.

\end{thebibliography}
\end{document}